\documentclass[12pt]{article}
\usepackage{amsfonts,amsmath,amssymb,amsthm,enumerate}
\def\dc{d_c}
\def\ball{\mathcal B_r(\mu)}
\def\pl{\phi_{\lambda}}

\def\tp{\Pi(\{\mu\},\ball)} 

\def\probsp{S}

\def\ps{\probsp}
\def\intp{\int_{\probsp\times\probsp}}

\def\prodsp{\probsp\times\probsp}

\def\PPP{\mathcal P}
\def\QQQ{\mathcal Q}
\def\XXX{\mathcal X}
\def\YYY{\mathcal Y}

\def\ints{\int_{\probsp}}
\def\intx{\int_{\mathcal X}}

\def\intspace{1.5pt}
\newcommand{\ip}[1]{\intp #1\hspace{\intspace} d\pi(x,y)}
\newcommand{\is}[1]{\ints #1\hspace{\intspace} d\mu(x)}
\newcommand{\intpp}[1]{\intp #1\hspace{\intspace} d\pi^*(x,y)}
\newcommand{\intsn}[1]{\ints #1\hspace{1.5pt} d\nu(x)}

\def\liminfk{\liminf_{k\to\infty}}
\newcommand\image[2]{#1_{\#}#2}
\newcommand{\projfirst}[1]{\image{(\textrm{proj}_1)}{#1}}
\newcommand{\projsecond}[1]{\image{(\textrm{proj}_2)}{#1}}

\def\pistar{\pi^{*}}
\def\villani{MR2459454}  
\def\vp{v_P}
\def\vd{v_D} 
\def\flc{f(y)+\lambda c(x,y)}

\def\RR{\mathbb R}

\def\NN{\mathbb N}
\def\EE{\mathbb E}

\newtheorem{definition}{Definition}
\newtheorem{theorem}{Theorem}
\newtheorem{lemma}{Lemma}

\newtheorem{remark}{Remark}
\newtheorem{corollary}{Corollary}

\providecommand{\keywords}[1]
{
  \small	
  \textbf{\textit{Keywords---}} #1
}

\title{A minimax approach to duality for linear distributional sensitivity testing}

\author{Gusti van Zyl\\ Department of Mathematics and Applied Mathematics\\ University of Pretoria}

\begin{document}

\maketitle
\keywords{duality; optimization, minimax; sensitivity testing; distributional robustness}
\begin{abstract}
	We consider the problem of finding the maximum of $\EE_{\nu}[f(X)]$ where $\nu$ is 
	allowed to vary over all the probability measures on a Polish space $S$ for which $d_c(\mu,\nu)\leq \theta$, in which $d_c$ is an optimal transport distance, $f$ a real-valued function on $S$ satisfying some regularity, $\mu$ a ``baseline" measure and $\theta \geq 0$. Whereas some of the derivations of the dual version of this optimization problem rely on Fenchel duality, we impose compactness on $S$ to allow us to instead use K. Fan's minimax theorem, which does not require vector space structure. This allows one to avoid the use of vector spaces of measures, or dual variables other than the Lagrange multiplier. 
\end{abstract}

\section{Introduction}  Let $(\ps,d)$ be a Polish (i.e. complete and separable) metric space and $\mu$ a measure, called the baseline measure or distribution. Consider the optimization problem 
\begin{eqnarray}\label{eq:maximization}
	\text{maximize} & \intsn{f(x)} \\
	\text{subject to } & \nu\in \ball,
\end{eqnarray}
where $\ball$, termed the uncertainty set, is a set of measures with transportation cost $\leq r$ from $\mu$. The transportation cost  derives from a cost function $c:\ps\times\ps\to\RR_+$. Assume $f$ is sufficiently regular such that a maximum exists and denote that maximum by $\vp$. Versions of the duality result 
\begin{equation}\label{eq:max-dual}
	\vp=\vd:=\inf_{\lambda \geq 0}\Big\{\lambda r+\EE_{\mu}[\sup_{y\in S}\{f(y)-\lambda c(X,y)\}]\Big\}
\end{equation}
are obtained under varying assumptions by amongst others Esfahani and Kuhn \cite{MohajerinEsfahani2018}; Gao and Kleywegt \cite{gao2022distributionally}, Blanchet and Murthy \cite{MR3959085}; Bartl, Drapeau and Tangpi \cite{doi:10.1111/mafi.12203}; and Feng and Schl\"ogl \cite{feng2018model}. This duality reduces the feasible set from an infinite-dimensional to a finite-dimensional problem. If, as often happens, the inner supremum in Equation (\ref{eq:max-dual}) is analytically  tractable, then it becomes a one-dimensional optimization problem that in itself may be analytically solvable, as can be seen in several examples in the above-mentioned references. To give just two examples, option prices robust to changes of a certain magnitude in the risk-neutral distribution \cite{MR4067077} and robust ruin probabilities \cite{MR3959085} have been calculated using such duality results.

Esfahani and Kuhn \cite{MohajerinEsfahani2018} assume that the baseline measure $\mu$ is an empirical distribution, consider the Wasserstein distance for $p=1$ only and assume a special structure for $f$ as a maximum of a finite number of convex functions. Gao and Kleywegt use $\liminf$ and $\limsup$ inequalities related to the set of  minimizers of $\inf_{y\in\ps}\{\lambda c(x,y)-f(y)\},$ for the case $c(x,y)=d(x,y)^p$ with $p \geq 1$. Their proof of duality is based on among others results analogous to Moreau-Yosida regularization. Feng and Schl\"ogl sketch a proof using a Karush-Kuhn-Tucker argument. Blanchet and Murthy work with general cost functions $c(x,y)$ and base their result on Fenchel duality. Their dual variables are the Lagrange parameter $\lambda$ and a measurable function $\phi$. To use Fenchel duality, they consider a vector space of bounded continuous functions on $\prodsp$ as well as a vector space of signed finite Borel measures on $\prodsp$, equipped with the variation norm.

In this paper we prove (\ref{eq:max-dual}) using K. Fan's minimax theorem instead of Fenchel duality. One benefit of this is that vector space structure is replaced by weaker convex-concavity requirements. Thus the minimax theorem allows one to proceed without having to introduce vector space structure, signed measures or different topologies. We use only the Lagrange parameter $\lambda$ as dual variable. The price we pay for this relative simplicity is that we assume $S$ is compact. We hope that in spite of restriction, this type of relatively simple proof may stimulate generalizations and also contribute to making the topic of distributional sensitivity testing more accessible to researchers already familiar with minimax arguments. It could also be possible to extend this argument to non-compact spaces, similar to the compact-to-general extension step in \cite{MR3959085}.

For the interpretation of the problem, comparison of optimal transport distance with Kulback-Leibler divergence, and various applications, see the above-mentioned literature. In this paper we restrict ourselves to the minimax proof of Equation (\ref{eq:max-dual}).

The structure of this paper is as follows: we define notation and recall the formulation of the problem in terms of transport plans, derive topological preliminaries allowing the application of minimax to the Lagrangian, and then apply minimax.

\section{Formulation in terms of transport plans}

For notational convenience we consider the problem obtained by replacing the maximization in (\ref{eq:maximization}) by minimization. Since the maximization problem for $f$ can be solved by solving the minimization problem for $-f$, there is no loss in doing so. 

To define the ball of measures which will be the feasible set of the optimization problem, we review a few facts about optimal transport. Let $\XXX$ and $\YYY$ be Polish spaces. If $\mu$ is any Borel measure on $\XXX$, and $T:\XXX\to\YYY$ a Borel map, then $T_{\#}\mu$ will denote the image measure defined by $(\image{T}{\mu})(A)=\mu(T^{-1}(A))$ for Borel sets $A\subseteq \YYY$. If $\gamma$ is a probability measure on $\XXX\times\YYY$, its marginal, or projection to $\XXX$ is the measure $\projfirst{\gamma}$ where $\textrm{proj}_1$ is the coordinate projection $\XXX\times\YYY\to\XXX:(x,y)\mapsto x.$ Equivalently $(\projfirst{\gamma})(A)=\gamma(A\times\ps)$ for each Borel set $A\subseteq \XXX.$ The marginal to $\YYY$ namely $\projsecond{\gamma}$ is defined similarly. Recall \cite[Definition 1.1]{MR2459454} that a transport plan, or a coupling, between a measure $\mu$ on $\XXX$ and a measure $\nu$ on $\YYY$ is a measure $\pi$ on $\XXX\times \YYY$ such that $\projfirst{\pi}=\mu$ and $\projsecond{\pi}=\nu.$  We denote the set of all Borel probability measures on a Polish space $\XXX$ with $P(\XXX)$, which is topologized by weak convergence of probability measures. If $\XXX$ is Polish then $P(\XXX)$ is Polish. In particular, the weak convergence of measures in $P(\XXX)$ is metrizable. 
If $\PPP$ and $\QQQ$ are sets of measures satisfying $\PPP\subseteq P(\XXX)$ and $\QQQ\subseteq P(\YYY)$, then the set of all transport plans from any $\mu\in\PPP$ to any $\nu\in\QQQ$ will be denoted by $\Pi(\PPP,\QQQ).$ A {\it cost} function is any lower semicontinuous (l.s.c.) $c:\XXX\times\YYY\to\RR_+$. A typical choice is $c(x,y)=d(x,y)^p$ for some $p \geq 1$. The {\it optimal transport cost} between measures $\mu,\nu$ on $\XXX$ is defined by $\dc(\mu,\nu):=\inf\{ \int_{\XXX\times \XXX} c(x,y) d\pi(x,y):\ \pi\in \Pi(\{\mu\},\{\nu\})\}.$ 

Let $(\probsp,d)$ be a Polish space and $f:\ps\to\RR$ be l.s.c. For any $\mu\in P(\ps)$ and $r>0$, we define
\begin{equation*}
	\ball:=\{\nu\in P(\ps):\ \dc(\mu,\nu) \leq r\}.
\end{equation*}  

We consider the optimization problem 
\begin{subequations}\label{eq:optimization}
\begin{align}	
	&\text{minimize }  \intsn{f(x)} \\
	&\text{subject to }  \nu\in \ball.
\end{align}  
\end{subequations}
(By using the word ``minimize" we do not imply that the minimum is attained.)
The set of transport plans $X:=\tp$ inherits the weak topology of probability measures from $P(\prodsp).$	

By translating the condition $\nu\in\ball$ to a condition on the transport plan one obtains, similarly to arguments in \cite[Section 3.1]{feng2018model}, that Problem (\ref{eq:optimization}) is equivalent, in the sense that the infimum agrees, to the following problem over transport plans:
\begin{subequations}\label{eq:primal}
\begin{align}
	&\text{minimize}  \ip{f(y)}, \\
	&\text{subject to }  \pi\in X \label{eq:c1} \text{ and }\\
	&\ip{c(x,y)}\leq r.\label{eq:c2}
\end{align}
\end{subequations}
We will refer to this as the primal problem, and the associated infimum value as $\vp$. 

\section{Compactness of transport plans when $\ps$ is compact}
Now we assume that $\ps$ is also compact. Then $P(\ps)$ is compact, for example by the Prokhorov theorem or \cite[Remark 6.19]{\villani}). 
Since $\ball$ is a closed subset of $S$, it too is compact.

Our main tool will be K. Fan's  minimax theorem \cite{MR55678} as formulated by Borwein and Zhuang \cite{MR838482}. 

\begin{definition}\cite{MR838482}\label{def:cclike} Let $X$ and $Y$ be sets, not necessarily having vector space structure. A function $K:X\times Y\to\RR$ is said to be {\it convex-concave like} on $X\times Y$ if for all $t$, $0\leq t\leq 1$, we have
	\begin{enumerate}[(a)]
		\item for all $x_1,x_2\in X$ there exists $x_3\in X$ such that for all $y\in Y$
		\begin{equation*}
			K(x_3,y)\leq tK(x_1,y)+(1-t)K(x_2,y);\text{ and }
		\end{equation*}
		\item for all $y_1,y_2\in Y$ there exists $y_3\in Y$ such that for all $x\in X$
		\begin{equation*}
			K(x,y_3)\geq tK(x,y_1)+(1-t)K(x,y_2).
		\end{equation*}
	\end{enumerate}
\end{definition}

\begin{theorem}\cite[Theorem A]{MR838482}\label{thm:minimax} Suppose that $X$ and $Y$ are non-empty sets with $K$ convex-concave like on $X\times Y$. Suppose that $X$ is compact and $K(\cdot,y)$ is l.s.c. on $X$ for each $y$ in $Y$. Then
	\begin{equation*}
		\min_{x\in X}\sup_{y\in Y} K(x,y)=\sup_{y\in Y}\min_{x\in X} K(x,y).
	\end{equation*} 
\end{theorem}
In our application $X=\tp$ is a set of transport -- sometimes called transference -- plans, which by application of the following theorem is compact in $P(\prodsp)$.

\begin{lemma}\cite[Corollary 5.21]{\villani}\label{le:compactness}
Let $\XXX$ and $\YYY$ be Polish spaces, and let $c(x,y)$ be a real-valued continuous cost function, $\inf c>-\infty$. Let $\mathcal K$ and $\mathcal L$ be two compact sets of $P(\XXX)$ and $P(\YYY)$ respectively. Then the set of optimal transference plans $\pi$ whose marginals respectively belong to $\mathcal K$ and $\mathcal L$ is itself compact in $P(\XXX\times \YYY)$.
\end{lemma}

We will need the following lemma, which is a variation of \cite[Lemma 4.3]{MR2459454} suitable for our purpose.  
\begin{lemma}\label{le:lsc} Let $\mathcal X$ and $\mathcal Y$ be Polish spaces.
	(1) If $g$ is a nonnegative l.s.c. real-valued function on $\mathcal X$ then the mapping $P(\mathcal X)\to\RR:\nu\mapsto \int_{\mathcal X} g(x)\ d\nu(x)$ is l.s.c. 
	
	(2) If $g$ is a nonnegative l.s.c. real-valued function on $\mathcal X\times \mathcal Y$ then the mapping $P(\mathcal X\times \mathcal Y)\to\RR:\gamma\mapsto \int_{\mathcal X\times \mathcal Y} g(x,y)\ d\gamma(x,y)$ is l.s.c. 
\end{lemma}
\begin{proof}
	(1) Let $\nu_k\to\nu$ weakly. 
	Since $g$ is l.s.c. and nonnegative, we can use the theorem of Baire to obtain a sequence $(g_n)_{n\in\NN}$ of continuous real-valued functions such that $0\leq g_n\uparrow g$. By replacing $g_n$ with $\min\{g_n,n\}$ if necessary, we can assume $g_n$ is bounded.
	By monotone convergence,
	\begin{equation}
		\intx g\ d\nu
		=\lim_{n\to\infty}\intx g_n\ d\nu
		= \lim_{n\to\infty}\lim_{k\to\infty}\intx g_n\ d\nu_k
		\leq  \liminfk \intx g\ d\nu_k.
	\end{equation}
	The proof of (2) is similar. 
\end{proof}

\section{Application of minimax}
Define the Lagrangian \begin{eqnarray}
	L(\pi,\lambda)&:=&\ip{f(x,y)}+\lambda\left(\ip{c(x,y)}-r\right)
	\label{eq:lagrangian}\\
	&=& \ip{f(x,y)+\lambda c(x,y)}-\lambda r.
\end{eqnarray} As is typical in the Lagrangian approach, the term 
$\sup_{\lambda \geq 0}\lambda\left(\ip{c(x,y)}-r\right)$ is infinity if constraint (\ref{eq:c2}) is not satisfied, and $0$ if it is satisfied. It follows that $\vp=\min_{\pi\in X} \sup_{\lambda \geq 0} L(\pi,\lambda)$.

\begin{theorem}\label{thm:main}
	The conditions of the minimax Theorem \ref{thm:minimax} are satisfied for $X$ as defined above, $Y=\RR_+$ and the function $K$ chosen as the Lagrangian $L$ of Equation (\ref{eq:lagrangian}). In particular:
	\begin{enumerate}
		\item $X$ is compact Hausdorff in the weak topology
		\item  $L:X\times Y\to\RR$ is l.s.c. on $X$ for every $\lambda \in Y$. 
		\item $L$ is convex-concave like. 
	\end{enumerate} 
\end{theorem}

\begin{proof}
(1) The weak compactness of $X$ follows from Lemma \ref{le:compactness}. Since the weak topology on $P(\prodsp)$ is metrizable, $P(\prodsp)$, and hence $X$, is Hausdorff.

(2) Fix $\lambda \geq 0.$ By Equation (\ref{eq:lagrangian}) and Lemma \ref{le:lsc} it is clear that the mapping $x\mapsto L(x,\lambda)$ is l.s.c. 

(3) Although $X$ and $Y$ are not vector spaces and $L$ is not linear, $L$ does preserve convex combinations: if $\pi_1$ and $\pi_2$ are transport measures belonging to $X$ and $0\leq t\leq 1$ then it is clear that for any $\lambda\in Y$,
		\begin{equation}
			L(t\pi_1+(1-t)\pi_2,\lambda)=tL(\pi_1,\lambda)+(1-t)L(\pi_2,\lambda).
		\end{equation}
	Therefore $\pi_3:=t\pi_1+(1-t)\pi_2$ yields equality in the first part of Definition \ref{def:cclike}. Equality in the second part of Definition \ref{def:cclike} is similar. 
\end{proof}

This enables us to derive the dual. The fact that the optimization over measures subproblem is solved by a measure concentrated on
$\{(x,y)\in S\times S:y\in \arg \min_{z\in S} \{f(z)+\lambda c(x,z)\}\}$ is observed, subject to obvious alterations to translate between minimization and maximization problems, amongst others in \cite{MR3959085} and \cite{feng2018model}. 

\begin{corollary}\label{co:formula}
	$\vp=\vd$, where
	\begin{equation}\label{eq:D}
		\vd:=\sup_{\lambda \geq0}\Big\{ \is{\psi_{\lambda}(x)}-\lambda r\Big\},
	\end{equation}
and $\pl(x):=\min_{y\in S}\{f(y)+\lambda c(x,y)\}.$
\end{corollary}
\begin{proof}
	By Theorem \ref{thm:main} we have 
	\begin{equation}\label{eq:minsup}
		\min_{\pi \in X}\sup_{\lambda \geq 0} L(\pi,\lambda)=\sup_{\lambda \geq 0 }\min_{\pi \in X} 
		L(\pi,\lambda).
	\end{equation}
	We have observed already that $\vp$ equals the LHS of Equation (\ref{eq:minsup}). For the RHS, fix $\lambda \geq 0$ and consider the inner minimization problem \begin{equation}\min_{\pi \in X} 
	L(\pi,\lambda)=\min_{\pi \in X}\Big\{\ip{f(x,y)+\lambda c(x,y)}\Big\}-\lambda r.\end{equation} 
	
	Since $X$ is compact and $L$ is l.s.c. on $X$ for every $\lambda \geq 0$, $L$ attains a minimum at say $\pi^{\star}\in X$. In fact we can construct the minimizer. Since $S$ is compact and the mapping $y\mapsto f(y)+\lambda c(x,y)$ is l.s.c., for each $x\in \ps$ the set of {\it integrand} minimizers $m(x):=\arg\min_{y\in S}\{f(y)+\lambda c(x,y)\}$ is non-empty. Let $T:S\to S$ be measurable and map $x$ to any element of $m(x).$ The existence of such a map follows from a measurable selection result in \cite{MR432846} and the details are given in Appendix A. This allows us to define the deterministic \cite[Definition 1.2]{MR2459454} coupling $\pistar:=\image{(Id,T)}{\mu}.$ It follows from this definition that the support  of $\pistar$ is a subset of $\{(x,T(x)):\ x\in S\}\subseteq S\times S.$

Consider an arbitrary transport plan $\pi\in X$. Combining the above-mentioned fact with the marginalization properties of the transport plans $\pistar$ and $\pi$, we get \begin{align*}
	&\intpp{\flc}  \\	
	= & \intpp{\min_{y\in S}\{\flc\}}\\
	= & \is{\min_{y\in S} \{\flc\}}\\
	= & \ip{ \min_{y\in S} \{\flc\}}\\
	\leq & \ip{\flc}.
	\end{align*}
	Since $\pi$ is arbitrary, $\pistar$ is a minimizer and 
	\begin{equation*}
		\min_{\pi\in X} L(\pi,\lambda)=L(\pistar,\lambda)=\ints \pl(x)\ d\mu(x)-\lambda r.
	\end{equation*}
Combining with Equation (\ref{eq:minsup}) we get Equation (\ref{eq:D}). 
\end{proof}

\begin{remark}
	A growth condition similar to that of \cite{MR3959085} or \cite{gao2022distributionally} will be needed in the extension of the result to non-compact $\ps$.
\end{remark}


\bibliographystyle{plain}
\bibliography{robust}

\section*{Appendix A: Measurable selection of minimizers in Corollary \ref{co:formula}}
Let $g(x,y)=f(y)+\lambda c(x,y)$ for $x,y\in S$. We need to show that a measurable function $T:S\to S$ that maps $x$ to $m(x):=\arg\min_{y\in S}\{f(y)+\lambda c(x,y)\}$ exists. Such problems have been studied in the statistics literature in connection with Bayes procedures and measurability of certain estimators. 

The result below is suitable to our purposes. Let $X,Y$ be metric spaces and let $D\subseteq X\times Y$. We use the notation $\mathrm{proj}(D)=\{x\in X:\text{there exists }y\in Y\text{ such that }(x,y)\in D\}$. We denote with $D_x\subseteq Y$ the section $\{y\in Y:\ (x,y)\in D\}$ and $F_x(y):=\inf_{y\in Y}F(x,y)$.

\begin{theorem}\cite[Corollary 1]{MR432846}
Let $X,Y$ be complete separable metric spaces and $F$ be a real-valued Borel measurable function defined on a Borel subset $D$ of $X\times Y$. 

Suppose that for each $x\in\mathrm{proj}(D)$, the section $D_x$ is $\sigma$-compact and $F(x,\cdot)$ is lower semi-continuous with respect to the relative topology on $D_x$. Then:
\begin{enumerate}[(i)]
	\item The sets 
	\begin{align*}
		G&=\mathrm{proj}(D),\\
		I&=\{x\in G:\ \text{for some }y\in D_x, F(x,y)=\inf F_x\},
	\end{align*}
are Borel.
\item For each $\epsilon>0$, there is a Borel measurable function 
$\phi_{\epsilon}$ satisfying, for $x\in G,$
\begin{align*}
f(x,\phi_{\epsilon}(x))&=\inf F_x, & \text{if }x\in I,\\
&\leq \epsilon+\inf F_x, &\text{if }x\notin I,\text{ and }\inf f_x\notin -\infty\\
&\leq -\epsilon^{-1}, &\text{if }x\notin I, \text{ and }\inf f_x=-\infty.
\end{align*}
\end{enumerate}
\end{theorem}
Applying this with $X=Y=S$, $F=g$, $\epsilon=1,\ D=\{(x,y)\in S\times S:\ x\in X,\ y\in m(x)\}$ we get a measurable function $T:=\phi_{\epsilon}:S\to S$ such that $F(x,T(x))=\inf F_x$.
The $\sigma$-compactness of $D_x$ follows from the fact that $g$ is l.s.c. in the second variable, so that $D_x=m(x)=F_x^{-1}\left((-\infty,\arg \min F_x]\right)$ is a closed subset of the compact space $S$, and hence compact. In our application $G=I=S$.

\end{document}